\documentclass[12pt]{amsart}

\usepackage{amsmath, amssymb}
\usepackage{amscd}
\usepackage{verbatim}
\usepackage{cleveref}
\usepackage{tikz}
\usetikzlibrary{arrows,chains,matrix,positioning,scopes}

\makeatletter
\tikzset{join/.code=\tikzset{after node path={%
\ifx\tikzchainprevious\pgfutil@empty\else(\tikzchainprevious)%
edge[every join]#1(\tikzchaincurrent)\fi}}}
\makeatother

\tikzset{>=stealth',every on chain/.append style={join},
         every join/.style={->}}

\newtheorem{definition}{Definition}[section]
\newtheorem{theorem}[definition]{Theorem}

\theoremstyle{definition}
\newtheorem{remark}[definition]{Remark}

\newtheorem{example}[definition]{Example}

\newcommand{\noi}{\noindent}

\newcommand{\F}{\mathcal{F}}

\newcommand{\ra}{\rightarrow}

\begin{document}

\title[Common Fixed Points of Semihypergroup Representations]{Common Fixed Points of Semihypergroup Representations}

\author[C.~Bandyopadhyay]{Choiti Bandyopadhyay}

\address{Department of Mathematics, SRM University AP, Amaravati, India}

\email{choiti@ualberta.ca, choiti.b@srmap.edu.in}


\begin{abstract}
In a series of previous papers \cite{CB1, CB2, CB3, CB4}, we initiated a systematic study of semihypergroups and had a thorough discussion on certain analytic and algebraic aspects associated to this class of objects. In particular, we introduced the notion of semihypergroup actions on a general topological space and discussed different continuity, equivalence and natural fixed point properties of the same in \cite{CB4}.   Now in this article, we consider different kinds of representations of a semihypergroup on compact convex subsets of a locally convex space and explore equivalence relations between certain fixed-point properties of such representations and amenability of the space of almost periodic functions. Finally, we investigate how far these equivalence relations can be strengthened when in particular, we consider representations on the dual of a Banach space.
\end{abstract}

\maketitle

\section{Introduction}
\label{intro}

\quad In the research of abstract harmonic analysis, left amenability of a given function-space is determined by the existence of left translation invariant means on the function-space under consideration. Since its inception, research on this topic has been a pivotal area of investigation in this area due to its vast applicability and deep interplay with other prominent areas of functional analysis, operator algebras and analysis in general. On the other hand, in the classical setting of topological semigroups and groups, fixed-point properties of their respective representations on certain subsets of a Banach space  or  a locally convex topological space, have been another pivotal area of research, imperative for an in-depth understanding of the underlying space in consideration.

It is well documented that invariance of this kind, i.e, left amenability is intrinsically connected to the fixed point theory of different kinds of actions of the underlying object on a certain topological space. In particular, the amenability properties of a (semi)topological semigroup are  intrinsically  related to the existence of  fixed points of certain actions of the semigroup on certain spaces, and in fact, one can completely  characterize amenability of different function-spaces of a semigroup using such fixed-point properties (see \cite{DA, LT, LZ, MIT, PA, RI} for example, among others). In a previous article \cite{CB4}, we introduced actions in the broader category of semihypergroups and initiated an in-depth discussion on the interplay between the existence of fixed points of different kinds of semihypergroup actions and left amenability of a semihypergroup, in order to thereby realise where and why this theory deviates from the classical theory of locally compact semigroups and groups. 

\quad A semihypergroup, as one would expect, can be perceived simply as a generalization of a locally compact semigroup, where the product of two points is a certain compact set, rather than a single point. In a nutshell, a semihypergroup is essentially a locally compact topological space where the measure space is equipped with a certain convolution product, turning it into an associative algebra (which unlike hypergroups, need not have an identity or an involution).

\quad The concept of semihypergroups, first introduced as `semiconvos' by Jewett\cite{JE}, arises naturally in abstract harmonic analysis in terms of left(right) coset spaces of locally compact groups, and orbit spaces of affine group-actions which appear frequently in different areas of research in mathematics, including Lie groups, coset theory, homogeneous spaces, dynamical systems, and  ordinary and partial differential equations, to name a few. These objects arising from locally compact groups, although retain some interesting structures from the parent category, often fail to be a topological group, semigroup or even a hypergroup (see \cite{CB1,JE} for detailed reasons). The fact that semihypergroups admit a broader range of examples arising from different fields of research compared to classical semigroup, group and hypergroup theory, and yet sustains enough structure to allow an independent theory to develop, makes it an intriguing and useful area of study with essentially a broader range of applications.
 

\quad However, unlike hypergroups, the category of (topological) semihypergroups still lack an extensive  systematic theory on it. In a series  of previous papers \cite{CB1, CB2, CB3, CB4} we initiated developing a systematic theory on semihypergroups. Our approach towards it has been to define and explore some basic algebraic and analytic aspects and areas on semihypergroups, which are imperative for further development of any category of analytic objects in general. 


It has been a long-standing open problem to characterize certain coset spaces, in general, semitopological semihypergroups, that have common fixed points whenever they act on a nonempty weak* compact convex subset of a locally convex space or the dual of a  Banach space as weak* continuous and norm nonexpansive mappings. In this article, we advance the theory by providing some complete characterizations of the left amenability of the space of almost periodic functions, in terms of existence of common fixed points of such representations ofthe semihyperhypergroup. The rest of the article is organized as the following.

\quad In the next, \textit{i.e.}, second section of this article, we recall some preliminary definitions and notations given by Jewett in \cite{JE}, and some notions used in \cite{CB4} required for further development. We conclude the section with listing some important examples of semihypergroups and hypergroups. 

 
\quad The third and final section contains proofs of the main results of the article. We begin with discussing the concepts of certain kinds of semihypergroup representations such as affine, equicontinuous, non-expansive and orbit-bounded, to name a few. We first show in Theorem \ref{mt1} that in this much broader setting, common fixed points of separately continuous, equicontinuous, affine actions on certain compact convex subsets completely characterize the left amenability of the function space of almost periodic functions.  Later in Theorem \ref{mt2} we prove that a complete characterization of left amenability of the this function space can also be carried out in terms of existence of common fixed points of separately continuous, non-expansive semihypergroup actions on certain compact convex sets. Finally in Theorem \ref{mt3} we investigate how the conditions of this equivalence can be strengthened if instead of actions on compact subsets of locally convex spaces, we consider actions on duals of Banach spaces. We see that a complete characterization of left amenability is indeed possible in this case, provided the action in consideration has at least one bounded orbit. We conclude the section with discussing certain immediate open questions in this area of research.





\section{Preliminary}
\label{Preliminary}

\noi We first list a preliminary set of notations and definitions that we will use throughout the text. All the topologies throughout this text are assumed to be Hausdorff.


\noi For any locally compact Hausdorff topological space $X$, we denote by $M(X)$ the space of all regular complex Borel measures on $X$, where $ M_F^+(X), M^+(X)$ and $P(X)$ respectively denote the subsets of $M(X)$ consisting of all non-negative measures with finite support, all finite non-negative regular Borel measures  and all probability measures on $X$. For any measure $\mu$ on $X$, we denote by $supp(\mu)$ the support of the measure $\mu$. Moreover, $ B(X), C(X), C_c(X)$ and $C_0(X)$ denote the function spaces of all bounded functions, bounded continuous functions, compactly supported continuous functions and continuous functions vanishing at infinity on $X$ respectively.


\noi Unless mentioned otherwise, the space $M^+(X)$ is equipped with the \textit{cone topology}   \cite{JE}, \textit{i.e,} the weak topology on $M^+(X)$ induced by $ C_c^+(X)\cup \{\chi_{_X}\}$. We denote  the set of all compact subsets of $X$ by $\mathfrak{C}(X)$, and consider the natural \textit{Michael topology} \cite{MT} on it, which makes it into a locally compact Hausdorff space.  For any element $x\in X$, we denote by $p_x$ the point-mass measure or the Dirac measure at the point $ x $.


\noi For any three locally compact Hausdorff spaces $X, Y, Z$, a bilinear map $\Psi : M(X) \times M(Y) \rightarrow M(Z)$ is called \textit{positive continuous} if the following properties hold true.
\begin{enumerate}
\item $\Psi(\mu, \nu) \in M^+(Z)$ whenever $(\mu, \nu) \in M^+(X) \times M^+(Y)$.
\item The map $\Psi |_{M^+(X) \times M^+(Y)}$ is continuous.
\end{enumerate}


\noi Now we state the formal definition for a (topological) semihypergroup. Note that we follow Jewett's notion \cite{JE} in terms of the definitions and notations, in most cases.


\begin{definition}[Semihypergroup]\label{shyper} A pair $(K,*)$ is called a (topological) semihypergroup if they satisfy the following properties:


\begin{description} 
\item[(A1)] $K$ is a locally compact Hausdorff space and $*$ defines a binary operation on $M(K)$ such that $(M(K), *)$ becomes an associative algebra.

\item[(A2)] The bilinear mapping $* : M(K) \times M(K) \rightarrow M(K)$ is positive continuous.

\item[(A3)] For any $x, y \in K$ the measure $(p_x * p_y)$ is a probability measure with compact support.

\item[(A4)] The map $(x, y) \mapsto \mbox{ supp}(p_x * p_y)$ from $K\times K$ into $\mathfrak{C}(K)$ is continuous.
\end{description}
\end{definition}


\noi Note that for any $A,B \subset K$ the convolution of subsets is defined as the following:
$$A*B := \cup_{x\in A, y\in B} \ supp(p_x*p_y)  .$$

\noi  We define the concepts of left (resp. right) topological and semitopological semihypergroups, in accordance with similar concepts in the classical semigroup theory.


\begin{definition}
A pair $(K, *)$ is called a left (resp. right) topological semihypergroup if it satisfies all the conditions of Definition \ref{shyper}, with property {($A2$)} replaced by property {($A2 '$)} (resp. property {($A2 ''$)}), given as the following:

\begin{description}
\item[(A2$'$)] The map $(\mu, \nu) \mapsto \mu*\nu$ is positive and for each $\omega \in M^+(K)$ the map\\ $L_\omega:M^+(K) \rightarrow M^+(K)$ given by $L_\omega(\mu)= \omega*\mu$ is continuous.\\
\item[(A2$''$)] The map $(\mu, \nu) \mapsto \mu*\nu$ is positive and for each $\omega \in M^+(K)$ the map\\ $R_\omega:M^+(K) \rightarrow M^+(K)$ given by $R_\omega(\mu)= \mu*\omega$ is continuous.
\end{description}
\end{definition}


A pair $(K, *)$ is called a \textit{semitopological semihypergroup} if it is both left and right topological semihypergroup, \textit{i.e}, if the convolution $*$ on $M(K)$ is only separately continuous.  

For any Borel measurable function $f$ on a (semitopological) semihypergroup $K$ and each $x, y \in K$, we define the left translate $L_xf$ of $f$ by $x$ (resp. the right translate $R_yf$ of $f$ by $y$) as
$$ L_xf(y) = R_yf(x) = f(x*y) := \int_K f \ d(p_x*p_y)\ .$$

Unless mentioned otherwise, we will always assume the uniform (supremum) norm $||\cdot ||_u$ on $C(K)$ and $B(K)$. We denote by $\mathcal{B}_1$ the closed unit ball of $C(K)^*$. Similarly, for any linear subspace $\F$ of $C(K)$, we denote the closed unit ball of $\F^*$ as $\mathcal{B}_1(\F^*):= \{\omega \in \F^*: ||\omega|| \leq 1\}$. Moreover, $\F$ is called left (resp. right) translation-invariant if $L_xf\in \F$ (resp. $R_xf\in \F$) for each $x\in K, f\in \F$. We simply say that $\F$ is translation-invariant, if it is both left and right translation-invariant.

A function $f\in C(K)$ is called left (resp. right) uniformly continuous if the map $x\mapsto L_xf$ (resp. $x\mapsto R_xf$) from $K$ to $(C(K), ||\cdot ||_u)$ is continuous. We say that $f$ is \textit{uniformly continuous} if it is both left and right uniformly continuous. The space consisting of  all such functions is denoted by $UC(K)$, which forms a norm-closed linear subspace of $C(K)$. 

The left (resp. right) orbit of a function $f\in C(K)$, denoted as $\mathcal{O}_l(f)$ (resp. $\mathcal{O}_r(f)$), is defined as $\mathcal{O}_l(f) := \{L_xf : x\in K\}$ (resp. $\mathcal{O}_r(f) := \{R_xf : x\in K\}$). A function $f\in C(K)$ is called left (resp. right) almost periodic if we have that $\mathcal{O}_l(f)$ (resp. $\mathcal{O}_r(f))$ is relatively compact in $(C(K), ||\cdot ||_u)$. We showed in a previous work \cite[Corollary 4.4]{CB1} that a function $f$ on $K$ is left almost periodic if and only if it is right almost periodic. Hence we  regard any left or right almost periodic function on $K$ simply as an \textit{almost periodic function}, and denote the space of all almost periodic functions on $K$ as $AP(K)$. We further saw in \cite{CB1} that $AP(K)$ is a norm-closed, conjugate-closed (with respect to complex conjugation), translation-invariant linear subspace of $C(K)$ containing constant functions, such that $AP(K)\subseteq UC(K)$.

\noi An element $e$ in a semihypergroup $K$  is called a \textit{two sided identity} of $K$ if  $$p_x * p_e = p_e * p_x = p_x$$ for any $x\in K$. Note that a two sided identity of a semihypergroup is necessarily unique \cite{JE}.


\begin{remark}
Given a Hausdorff topological space $K$, in order to define a continuous bilinear mapping $* : M(K) \times M(K) \rightarrow M(K)$, it suffices to only define the measures $(p_x*p_y)$ for each $x, y \in K$.
 This is true since one can then extend  the convolution `$*$' bilinearly to $M_F^+(K)$. As $M_F^+(K)$ is dense in $M^+(K)$ in the cone topology \cite{JE}, one can further achieve a continuous extension of  `$*$' to $M^+(K)$ and hence to the whole of $M(K)$ using bilinearity.
\end{remark}


\vspace{0.03in}



\noi Now we list some well known examples \cite{JE,ZE} of semihypergroups and hypergroups. See \cite[Section 3]{CB1} for details on the constructions as well as the reasons why most of the structures discussed there, although attain a semihypergroup structure, fail to be hypergroups. 


\begin{example} \label{extr}

If $(S, \cdot)$ is a locally compact topological semigroup, then $(S, *)$ is a semihypergroup where $p_x*p_y = p_{_{x.y}}$ for any $x, y \in S$. Similarly, if $(G, \cdot)$ is a locally compact topological group with identity $e_G$, then $(G, *)$ is a hypergroup with the same bilinear operation $*$, identity element $e_G$  and the involution on $G$ defined as $x \mapsto x^{-1}$.

Note that $Z(S) =S$, $Z(G) = G$.
\end{example}

\begin{example} \label{ex2}
Take $T = \{e, a, b\}$ and equip it with the discrete topology. Define
\begin{eqnarray*}
p_e*p_a &=& p_a*p_e \ = \ p_a\\
p_e*p_b &=& p_b*p_e \ = \ p_b\\
p_a*p_b &=& p_b*p_a \ = \ z_1p_a + z_2p_b\\
p_a*p_a &=& x_1p_e + x_2p_a + x_3p_b \\
p_b*p_b &=& y_1p_e + y_2p_a + y_3p_b
\end{eqnarray*}

\noi where $x_i, y_i, z_i \in \mathbb{R}$ such that $x_1+x_2+x_3 = y_1+y_2+y_3 = z_1+z_2 = 1$ and $y_1x_3 = z_1x_1$. Then $(T, *)$ is a commutative hypergroup with identity $e$ and the identity function on $T$ taken as involution. In fact, any finite set can be given several (not necessarily equivalent)  semihypergroup and hypergroup structures.
\end{example}

\begin{example} \label{ex3}
Let $G$ be a locally compact group and $H$ be a compact subgroup of $G$ with normalized Haar measure $\mu$. Consider the left coset space $S := G/H = \{xH : x \in G\}$ and the double coset space $K := G/ /H = \{HxH : x \in G\}$ and equip them with the respective quotient topologies. Then $(S, *)$ is a semihypergroup and $(K, *)$ is a hypergroup where the convolutions are given as following for any $x, y \in G$ :
$$p_{_{xH}} * p_{_{yH}} = \int_H p_{_{(xty)H}} \ d\mu(t), \ \ \ \ p_{_{HxH}} * p_{_{HyH}} = \int_H p_{_{H(xty)H}} \ d\mu(t)  .$$
It can be checked \cite{CB1} that the coset spaces $(S, *)$ fail to have a hypergroup structure. 
\end{example}

\begin{example} \label{ex4}
Let $G$ be a locally compact topological group and $H$ be any compact group with normalized Haar measure $\sigma$. 
For any continuous affine action\cite{JE} $\pi$ of $H$ on $G$, consider the orbit space $\mathcal{O} := \{x^H : x \in G\}$, where for each $x\in G$,  $x^H =  \{\pi(h, x): h\in H\}$ is the orbit of $x$ under the action $\pi$. 


\noi Consider $\mathcal{O}$ with the quotient topology and the following convolution:
$$ p_{_{x^H}} * p_{_{y^H}} := \int_H \int_H p_{_{(\pi(s, x)\pi(t, y))^H}} \ d\sigma(s) d\sigma(t) .$$
\noi Then $(\mathcal{O}, *)$ becomes a semihypergroup. It can be shown \cite{BH, JE} that $(\mathcal{O}, *)$ becomes a hypergroup only if for each $h\in H$, the map $x\mapsto \pi(h, x) : G\ra G$ is an  automorphism.
\end{example}


\section{Main Results}

Recall \cite{CB3} that a semihypergroup action $\pi$ of $K$ on a locally convex Hausdorff topological vector space $E$ (also denoted by $(E, Q)$ where $Q$ is the family of  seminorms inducing the topology on $E$) is a map $\pi: K\times E\ra E$ such that the following conditions are satisfied:
\begin{enumerate}
\item For any $s, t \in K$, $x\in E$ we have 
$$\pi(s, \pi(t, x)) = \int_K \pi(\zeta,x) \ d(\delta_s*\delta_t)(\zeta).$$
\item Whenever $K$ has a two-sided identity $e$, for each $x\in E$  we have that $$\pi(e, x) =x \ .$$
\end{enumerate}

Such an action $\pi$ of $K$ on $E$ is called  \textit{(separately) continuous}, if the map $\pi$ is (separately) continuous on $K\times E$. For each $s\in K$, we denote by $\pi_s$ the co-efficient map $$x\mapsto \pi(s, x):E\ra E \ .$$ 

An action $\pi$ of $K$ is called \textit{affine} if $\pi_s$ is an affine map for each $s\in K$.  We say that some $x_0\in E$ is a \textit{common fixed point} of $\pi$ if $\pi_s(x_0)=x_0$ for each $s\in K$. In fact, using the fact that commutative semihypergroups are amenable \cite{CB1}, one can prove the existence of common fixed points for such semihypergroups. 

\begin{theorem} [\cite{CB3}] \label{main2}
Let $K$ be commutative  and $C$ be a compact convex subset of a locally convex Hausdorff vector space $(E, Q)$. Then any separately continuous affine action of $K$ on $C$ has a fixed point.
\end{theorem}

On the other hand, a stricter  fixed point property  serves as a sufficient condition for any semihypergroup $K$ to admit a LIM on $AP(K)$. 

\begin{theorem} [\cite{CB3}] \label{main3}
Assume that any jointly continuous affine action $\pi$ of $K$ on a compact convex subset $C$ of a locally convex vector space $(E, Q)$ has a common fixed point. Then there exists a LIM on $AP(K)$.
\end{theorem}

In what follows, we see that although the intrinsically different structures of the underlying space and function-spaces in this broader setting warrants further caution and evolved Banach algebraic techniques, one can indeed extend a number of useful pivotal fixed-point characterizations of left amenability to this broad setting of semihypergroups, from the classical setting of  semigroups \cite{LA1, LZ}. 

Recall that given a compact subset $C$ of a separated locally convex space $E$, the subspace topology on $C$ is induced by a unique uniformity $\mathcal{U}$ corresponding to $C$ \cite[Chapter 6]{KE}, namely the family of all neighborhoods of the diagonal $\Delta= \{(x, x): x\in C\}$. We say that an action $\pi$ of $K$ on $C$ is \textit{equicontinuous} if for each $x\in C$ and $U\in \mathcal{U}$, there exists some $V\in \mathcal{U}$ such that $(\pi(s,x), \pi(s, y))\in U$ for each $s\in K$ whenever $(x, y)\in V$. An important subclass of such actions consists of the non-expansive actions. An action $\pi$ of $K$ on $(E, Q)$ is called \textit{non-expansive} if we have that $p(\pi(s, x) - \pi(s, y))\leq p(x - y)$ for each $p\in Q, s\in K$ and any $x, y\in E$.

Given a separately continuous action $\pi$ of $K$ on a compact Hausdorff space $X$, for each $y\in X$ and $f\in C(X)$ we define a map $T_y^\pi f : K \ra \mathbb{C}$ as $$T_y^\pi f(s):= f(\pi(s, y)) = f\circ \pi_s(y),$$ for each $s\in K$. It is easy to see that $T_y^\pi f$ is a continuous function. Moreover, it follows from the continuity of $f$ and $\pi_s$ that the orbit of the function $T_y^\pi f$ is relatively compact in $C(K)$, i.e, we will always have that $T_y^\pi f \in AP(K)$.

Now we proceed to deriving several pivotal characterizations of amenability of the function space $AP(K)$ in terms of the existence of common fixed points of such actions or representations.

\begin{theorem} \label{mt1}
$AP(K)$ is left amenable if and only if every separately continuous, equicontinuous affine action of $K$ on a compact convex subset $C$ of a separated locally convex space $E$ has a common fixed point.
\end{theorem}

\begin{proof}
First, assume that any separately continuous, equicontinuous affine action $\pi$ of $K$ on a compact convex subset $C$ of a separated locally convex space $E$ has a common fixed point.

Recall \cite{CB1} that the function $AP(K)$ is both left and right translation invariant. For each $s\in K$, consider the dual $L_s^*$ of the Left translation operator $L_s:AP(K)\ra AP(K)$, where $L_s^*: AP(K)^*\ra AP(K)^*$ is canonically defined as $$L_s^*\phi(f) = \phi(L_sf)$$ for each $\phi\in AP(K)^*, f\in AP(K)$. Let $C\subseteq AP(K)^*$ denote the set of all means on $AP(K)$ and $E:= AP(K)^*$. We equip $E$ with the weak$^*$ topology and define a map $\pi: K\times E \ra E$ by $$\pi(s, \phi) := L_s^*\phi$$
for each $s\in K, \phi\in E$. We know from \cite[Theorem 4.4]{CB4} that $\pi$ is an affine, separately continuous action of $K$ on $E$, and $C$ is weak$^*$-compact in $E$. Now to check if $\pi$ is equicontinuous, for each $f\in AP(K)$ let $p_f$ be the seminorm on $AP(K)^*$ given by $$p_f(\phi) := sup \{|\phi(L_sf)| : s\in K\}$$ for each $\phi\in AP(K)^*$. Set $Q:= \{p_f : f\in AP(K)\}$. 

Let $\varepsilon>0$ be given. Given $f\in AP(K)$, choose $\phi, \psi\in AP(K)^*$ such that $$sup_{s\in K} |(\phi - \psi)(L_sf)| = p_f(\phi-\psi) < \varepsilon .$$ Then for each $x\in K$ and any $\phi, \psi\in AP(K)^*, f\in AP(K)$ we have that
\begin{eqnarray*}
p_f(L_x^*\phi - L_x^*\psi)&=& sup_{s\in K} |(L_x^*\phi - L_x^*\psi)(L_sf)|\\
&=& sup_{s\in K} |(\phi-\psi)(L_x(L_sf))|\\
&=& sup_{s\in K} \{ \Big{|} \int_K (\phi-\psi) (L_uf) \ d(p_x*p_s)(u) \Big{|}\\
&\leq & sup_{s\in K} \int_K |(\phi-\psi)(L_uf)| \ d(p_x*p_s)(u)\\
&\leq & sup_{s\in K} \int_K p_f(\phi-\psi) \ d(p_x*p_s)(u)\\
&< &  \varepsilon \ (p_x*p_s)(K) \ = \ \varepsilon,
\end{eqnarray*}
as required. Now in particular, the given fixed point property provides us with a common fixed point $m_0\in C$ of $\pi$, i.e, we get a mean $m_0$ on $AP(K)$ such that $L_s^*m_0 = m_0$ for each $s\in K$. Hence $m_0$ serves as a LIM on $AP(K)$.

\vspace{0.1in}

Now conversely, let $AP(K)$ be left amenable and $m$ be a LIM on $AP(K)$. Furthermore, let $\pi$ be a separately continuous, equicontinuous, affine action of $K$ on a compact convex subset $C$ of a separated locally convex space $(E, Q)$. We denote by $A_f(K)$ the closed subspace of $C(K)$ consisting of all real valued continuous affine functions on $K$. For each $s\in K$, the evaluation map $E_s\in C(K)^*$ is canonically defined as $E_s(f) := f(s)$ for each $f\in C(K)$.

Since $m$ is a mean and the evaluation maps serve as extreme points of the unit ball, we know \cite{CB4} that there exists a net $\{\psi_\alpha\}$  of finite means such that $$\lim_\alpha \psi_\alpha(f) = m(f)$$ for each $f\in AP(K)$. Here for each $\alpha$ there exists some $n_\alpha\in \mathbb{N}$ such that $$\psi_\alpha = \sum_{i=1}^{n_\alpha} \lambda_i^\alpha E_{s_i^\alpha}$$ where each $s_i^\alpha\in K$ and each $\lambda_i^\alpha >0$ such that $\sum_{i=1}^{n_\alpha} \lambda_i^\alpha =1$. Now fix some $y\in C$ and consider the net $\{z_\alpha\}\subseteq C$ where for each $\alpha$ we set $$z_\alpha := \sum_{i=1}^{n_\alpha} \lambda_i^\alpha \pi(s_i^\alpha, y) .$$ Since $C$ is compact, let $z$ be a cluster point of $\{z_\alpha\}$ in $C$. On the other hand, for each bounded continuous function $f$ on $C$, consider the map $T_yf: K \ra \mathbb{C}$ given by $$T_yf(s):= f(\pi(s, y)) = (f\circ \pi_s) (y).$$ It follows from the continuity of $f$ that $T_yf\in AP(K)$ (we provide a general proof of this fact while discussing the proof of the next main theorem, as well). Hence in particular, passing through a subnet if necessary, without loss of generality, for each function $h\in A_f(C)$ we have that
\begin{eqnarray*}
h(z) \ = \ h(\lim_\alpha \ z_\alpha) &=& \lim_\alpha h\Big{(}\sum_{i=1}^{n_\alpha} \lambda_i^\alpha \pi(s_i^\alpha, y)\Big{)}\\
&=& \lim_\alpha \sum_{i=1}^{n_\alpha} \lambda_i^\alpha h(\pi(s_i^\alpha, y))\\
&=& \lim_\alpha \sum_{i=1}^{n_\alpha} \lambda_i^\alpha T_yh(s_i^\alpha)\\
&=& \lim_\alpha \sum_{i=1}^{n_\alpha} \lambda_i^\alpha E_{s_i^\alpha} (T_yh)\\
&=& \lim_\alpha \psi_\alpha(T_yh)\\
&=& m(T_yh) \ = \ m\big{(}L_s(T_yf)\big{)}
\end{eqnarray*}
for each $s\in K$, where the second last equality holds true since $T_yh\in AP(K)$ and the last equality follows from the left translation invariance of the mean $m$. Hence for each $s\in K, f\in A_f(C)$ we attain the following invariance:
\begin{eqnarray*}
h(z) \ = \ m\big{(}L_s(T_yh)\big{)} &=& \lim_\alpha \psi_\alpha \big{(}L_s(T_yh)\big{)}\\
&=& \lim_\alpha \sum_{i=1}^{n_\alpha} \lambda_i^\alpha E_{s_i^\alpha} \big{(}L_s(T_yh)\big{)}\\
&=& \lim_\alpha \sum_{i=1}^{n_\alpha} \lambda_i^\alpha L_s(T_yh)(s_i^\alpha)\\
&=& \lim_\alpha \sum_{i=1}^{n_\alpha} \lambda_i^\alpha \int_K T_yh(\zeta) \ d(p_s*p_{s_i^\alpha}) (\zeta)\\
&=& \lim_\alpha \sum_{i=1}^{n_\alpha} \lambda_i^\alpha \int_K h(\pi_\zeta(y)) \ d(p_s*p_{s_i^\alpha}) (\zeta)\\
&=& \lim_\alpha \sum_{i=1}^{n_\alpha} \lambda_i^\alpha h \Big{(} \int_K \pi_\zeta(y)\ d(p_s*p_{s_i^\alpha}) (\zeta)\Big{)}\\
&=&  h\Big{(} \lim_\alpha \sum_{i=1}^{n_\alpha} \lambda_i^\alpha \int_K \pi_\zeta(y)\ d(p_s*p_{s_i^\alpha}) (\zeta)\Big{)}\\
&=& h\Big{(} \lim_\alpha \sum_{i=1}^{n_\alpha} \lambda_i^\alpha (\pi_s\circ \pi_{s_i^\alpha}) (y)\Big{)}\\
&=& h\Big{(} \pi_s \big{(}\lim_\alpha \sum_{i=1}^{n_\alpha} \lambda_i^\alpha  \pi_{s_i^\alpha} (y)\big{)}\Big{)}\\
&=& h\Big{(} \pi_s (\lim_\alpha \ z_\alpha)\Big{)} \ = \ h(\pi(s, z)),
\end{eqnarray*}
where the seventh and eighth equality follows dince $h$ is affine and continuous respectively and the third last equality follows since $\pi$ is separately continuous and affine. But note that $A_f(C)$ separates the points of $C$. Hence we must have that $\pi(s, z) = z$ for each $s\in K$, as required
\end{proof}

\begin{theorem} \label{mt2}
Let $K$ be a semitopological semihypergroup. Then $AP(K)$ is left amenable if and only if every separately continuous, non-expansive action of $K$ on a compact convex subset $C$ of a separated locally convex space $E$ has a common fixed point.
\end{theorem}

\begin{proof}
First, let $\psi$ be a LIM on $AP(K)$, and $\pi$ be a separately continuous, non-expansive action of $K$ on a compact convex subset $C$ of a separated locally convex space $E$. 

It follows from a standard application of Zorn's lemma that there exists a  $X\subseteq C$ where $X$ is a non-empty minimal closed, convex, $\pi$-invariant subset. Furthermore, using the same technique for the set $X$, we can find a non-empty minimal closed $\pi$-invariant subset $F\subseteq X\subseteq C$. We fix an element $y\in F$ and consider the map $\phi: C(F) \ra \mathbb{C}$ given by $$\phi(f) := \psi(T_y^\pi f)$$ for each $f\in C(F)$. Note that the map is well defined since $T_y^\pi f \in AP(K)$ as each $f$ is continuous. Since $\psi$ is a mean on $AP(K)$, it follows by construction that $\phi$ is a mean on $C(F)$ as well.

Now for each $s\in K$, $f\in C(F)$ we define the map $N_sf:F\ra \mathbb{C}$ given by $$ N_sf(x) := f(\pi(s, x)) = f\circ \pi_s(x) = T_x^\pi f (s),$$
for each $x\in F$. It is immediate that $N_sf$ is continuous since both $f, \pi_s$ are continuous. Furthermore, note that  $T_y^\pi N_sf = L_sT_y^\pi f$ for any $s\in K$ since for any $s, t\in K$ we have that
\begin{eqnarray*}
T_y^\pi N_sf (t) &=& N_sf(\pi(t, y))\\
&=& f\Big{(}\pi(s, \pi(t, y))\big{)}\\
&=& \int_K f\big{(}\pi(\zeta, y)\big{)} \ d(p_s*p_t)(\zeta)\\
&=& \int_K T_y^\pi f(\zeta) \ d(p_s*p_t)(\zeta)\\
&=& T_y^\pi f(s*t) \ = \ L_sT_y^\pi f(t).
\end{eqnarray*}

Hence for each $s\in K$, $f\in C(F)$ we have the following equality:
\begin{eqnarray*}
\phi(N_sf) &=& \psi(T_y^\pi N_sf)\\
&=& \psi (L_s T_y^\pi f)\\
&=& \psi(T_y^\pi f) \ = \ \phi(f),
\end{eqnarray*}
where the third equality follows since $\psi$ is left invariant on $AP(K)$. Now, let $\lambda$ be the probability measure on $F$ such that $$\phi(f) = \int_F f \ d\lambda$$
for each $f\in C(F)$. For each $s\in K$ and any Borel measurable set $B\subset F$ consider the set $$\tilde{s}B := \{x\in F:  \pi(s, x)\in B\}.$$ Note that $\tilde{s}B$ will also be a Borel set in $F$ since $\pi_s$ is continuous, and hence induces a Borel measure $\tilde{s}\lambda$ on $F$ for each $s\in K$ defined as $$\tilde{s}\lambda (B):= \lambda(\tilde{s}B) = \lambda\big{(}\pi_s^{-1}(B)\big{)}.$$ In particular, for each $s\in K$ and any continuous function $f$ on $F$, since $\phi(N_sf) = \phi(f)$, we have that
\begin{eqnarray*}
 \int_F f \ d\lambda &=& \int_F N_sf \ d\lambda\\
 &=& \int_F f\circ \pi_s (x) \ d\lambda(x)\\
 &=& \int_F f(z) \ d(\lambda(\pi_s^{-1}(z)) \ = \ \int_F f \ d(\tilde{s}\lambda).
\end{eqnarray*}
Hence for each Borel set $B\subseteq F$ and any $s\in K$ we have that $\tilde{s}\lambda(B) = \lambda(B) = \lambda(\pi_s^{-1}(B))$.

Now, let $\mathcal{F}$ be the family of all closed subsets $A\subseteq F$ such that $\lambda(A)=1$. Consider the set $$F_0 := \bigcap_{A\in \mathcal{F}} A = supp(\lambda).$$ In particular, for any $s\in K$ since $\tilde{s}A\in \mathcal(F)$ for each $A\in \mathcal{F}$, we have that $F_0 \subseteq \pi_s^{-1}(F_0)$, i.e, $\pi_s(F_0)\subseteq F_0$. Hence by minimality of the closed $\pi$-invariant subset $F\subseteq C$, we have that $$F= F_0 = supp(\lambda).$$ On the other hand, for each $s\in K$ we have that $\lambda(\pi_s(F)) = \lambda(\pi_s^{-1}(\pi_s(F))) = \lambda (F) =1$. Hence $\pi_s(F) \in \mathcal{F}$, i.e, $F\supseteq \pi_s(F)\supseteq F_0 = F$.

Thus we have that $\pi_s(F) = F$ for each $s\in K$. Hence if $F$ is singleton, it serves as the required common fixed point of $\pi$. We claim that $F$ must be singleton in this case, and prove this by contradiction. 

If possible, let $F$ contains at least two distinct points. Then there exists a continuous seminorm $p\in Q$ such that $ r:= sup \{p(x-y): x, y\in F\}>0$. Then there exists some $z_0\in \bar{co}F$ such that $$0<r_0:= sup\{p(z_0-x): x\in F\} <r .$$
Now set $X_0 := X\cap \Big{(} \cap_{x\in F} B_p[x, r_0]\Big{)}$, where $B_p[x, r_0]$ denotes the closed $p$-ball $\{z\in X : p(z-x)\leq r_0\}$. Since $z_0\in X_0$ and $r_0<r$, we see that $X_0$ is a non-empty closed, convex, proper subset of $X$. 

In particular, $F\subseteq B_p[x, r_0]$ for each $x\in X_0$ since $p(z-x)\leq r_0$ for each $z\in F$ by construction of $X_0$. Since $\pi$ is non-expansive, for any $s\in K$, $x\in X_0, z\in F$ we have that $$p(\pi(s, x)- \pi(s, z))\leq p(x-z) \leq r_0 .$$ Therefore $F= \pi_s(F) \subseteq B_p[\pi(s, x), r_0]$, and hence $p(z-\pi(s, x)) \leq r_0$ for each $z\in F$, $s\in K, x\in X_0$. 

Hence finally, for any $x\in X_0$ we have that $\pi_s(x)\in B_p[z, r_0]$ for each $z\in F$, i.e, the closed, convex set $X_0$ is $\pi$-invariant as well. This contradicts the minimality of $X$ since $X_0\subseteq X$ is a proper subset. Therefore our claim that $F$ must be singleton is true, providing us with a common fixed point of the action $\pi$.

The converse of the statement follows similar lines as the proof pf the respective direction of the statement in Theorem \ref{mt1}.

\end{proof}

These fixed-point properties can further be improved by taking the locally convex space to be a dual Banach space $X^*$ in particular, equipped with the weak$^*$ topology.

\begin{theorem} \label{mt3}
$AP(K)$ is left amenable if and only if every separately continuous, equicontinuous, affine action of $K$ on a dual Banach space $X^*$ which admits an element $\omega_0\in X^*$ such that the orbit $\{\pi(s, \omega_0): s\in K\}$ is bounded,  has a common fixed point.
\end{theorem}

\begin{proof}
First, assume that any separately continuous, equicontinuous, affine action $\pi$ of $K$ on a dual Banach space $X^*$ for which we can find   an element $\omega_0\in X^*$ such that the orbit $$\mathcal{O}(\omega_0)= \{\pi(s, \omega_0): s\in K\}$$ is bounded,  admits a common fixed point of the action $\pi$.

Given any $\mu\in M(K), f\in C(K)$, consider the averaging function $$L_\mu f := \int_K L_xf \ d\mu(x) .$$ Then for any $s, t\in K$ we have that 
\begin{eqnarray*}
R_t(L_\mu f) (s) &=& L_\mu f (s*t)\\
&=& \int_K L_\mu f(u) \ d(p_s*p_t)(u)\\
&=& \int_K\int_K L_x f(u) \ d\mu(x) \ d(p_s*p_t)(u)\\
&=& \int_K \int_K f(x*u) \ d(p_s*p_t)(u) \ d\mu(x) \\
&=& \int_K L_xf (s*t) \ d\mu(x) \ = \ \int_K R_t(L_xf)(s)  \ d\mu(x). 
\end{eqnarray*}
Hence for any $t\in K$ we have that $$R_t(L_\mu f) = \int_K R_t(L_xf) \ d\mu(x) = \int_K L_x(R_t f) \ d\mu(x) = L_\mu (R_tf). $$ Thus we see that if $f\in AP(K)$, then so is $L_\mu f$ for each $\mu\in M(K)$ since the map $f\mapsto L_\mu f: C(K)\ra C(K)$ is continuous, as for any  $f, g\in C(K)$ we have that 
\begin{eqnarray*}
|L_\mu f(s) - L_\mu g(s)| &=& \Big{|}\int_K L_x(f-g)(s) \ d\mu(x)\Big{|}\\
&\leq & \int_K |L_x(f-g)(s)| \ d|\mu |(x)\\
&\leq & \int_K ||L_x(f-g)||_\infty \ d|\mu|(x)\\
&\leq & ||f-g||_\infty |\mu|(K).
\end{eqnarray*} 

Now we consider the map $\Phi: M(K)\times AP(K) \ra AP(K)$ given by $$(\mu, f) \mapsto L_\mu f := \int_K L_xf \ d\mu(x) $$ for each $\mu\in M(K), f\in AP(K)$. We claim that $\Phi$ induces an $M(K)$-module action on $AP(K)$. To prove this, we pick any two elements $\mu, \nu\in M(K)$ and note that for any $s\in K$ we have that 
\begin{eqnarray*}
L_{\mu*\nu} f(s) &=& \int_K L_xf(s) \ d(\mu*\nu)(x)\\
&=& \int_K R_sf(x) \ d(\mu*\nu)(x)\\
&=& \int_K \int_K R_sf(x*z) \ d\mu(x) \ d\nu(z)\\
&=& \int_K \int_K L_xf(z*s) \ d\mu(x) \ d\nu(z)\\
&=& \int_K L_z(L_\mu f)(s) \ d\nu(z)\\
&=& L_\nu (L_\mu f)
\end{eqnarray*}
where the fourth equality follows since $f(x*y*z)= L_xf(y*z) = R_zf(x*y)$ for any $x, y , z\in K$ and any bounded measurable function $f$ on $K$. Moreover, the second last equality follows since
\begin{eqnarray*}
L_z(L_\mu f)(s) \ = \ L_\mu f(z*s) &=& \int_K L_\mu f (t) \ d(p_z*p_s)(t)\\
&=& \int_K\int_K L_xf(t) \ d\mu(x) \ d(p_z*p_s)(t)\\
&=& \int_K L_xf(t)  \ d(p_z*p_s)(t) \ d\mu(x) \ = \ \int_K L_xf(z*s) \ d\mu(x).
\end{eqnarray*}
Hence $E:=AP(K)$ is indeed a right $M(K)$-module with the module action $\Phi: \mu\mapsto L_\mu$. Now for any scalar $\alpha\in \mathbb{C}$, let $\alpha \mathbf{1}$ denote the constant function taking $\alpha$ on $K$. Then 
\begin{eqnarray*}
 L_\mu(\alpha \mathbf{1})(s) &=& \int_K L_x(\alpha \mathbf{1})(s) \ d\mu(x)\\
 &=& \int_K (\alpha \mathbf{1})(x*s) \ d\mu(x)\\
 &=& \int_K \int_K (\alpha \mathbf{1})(z) \ d (p_x*p_s) \ d\mu(x)\\
 &=& \alpha \int_K (p_x*p_s)(K) \ d\mu(x) \ = \ \alpha \mu(K).
 \end{eqnarray*}
Thus $\mathbb{C}\mathbf{1}$ is a submodule of $E$, and hence the quotient $X:= E/\mathbb{C}\mathbf{1}$ is a Banach right $M(K)$-module. Now we consider the Banach left $M(K)$-module $X^*$. Note that we can identify $X^*$ with the left submodule $\mathbb{C}\mathbf{1}^{\perp}$ of $E^*$ given by $$\mathbb{C}\mathbf{1}^{\perp} := \{u\in E^*: u(\mathbf{1}) = 0\}.$$
where the standard left module action is given by $\mu \mapsto L_\mu^*$. 

Next, we equip $X^*$ with the weak$^*$ topology and consider the map $\Psi: M(K)\times (X^*, weak^*) \ra (X^*, weak^*)$ given by $$\Psi(\mu, \phi) := L_\mu^*\phi .$$ In particular, for each $\mu\in M(K), \phi\in X^*, f\in X$ we have that $L_\mu^*\phi(f) = \phi(L_\mu f)$. Now for any $\mu, \nu\in M(K)$ and $\phi, \psi\in X^*$ the following holds true for each $f\in X$:
\begin{eqnarray*}
\Big{|} \Big{(}\Psi(\mu, \phi) - \Psi(\nu, \psi)\Big{)}(f)\Big{|} &=& \Big{|} \phi(L_\mu f) - \psi(L_\nu f)\Big{|}\\
&\leq & \Big{|} \phi(L_\mu f) -\phi(L_\nu f) \Big{|} +  \Big{|} \phi(L_\nu f) - \psi(L_\nu f)  \Big{|}\\
&\leq & ||\phi|| \ ||L_\mu f - L_\nu f ||_\infty + ||\phi-\psi|| \ ||L_\nu f||_\infty\\
&\leq & ||\phi|| \ ||f||_\infty \ ||\mu - \nu|| + ||\phi-\psi|| \ ||f||_\infty \ ||\nu|| \\
 &=& \ \Big{(}||\phi|| \  ||\mu-\nu|| + ||\nu|| \ ||\phi-\psi|| \Big{)} ||f||_\infty,
\end{eqnarray*}
where the third inequality follows since for each $s\in K$ we have $$|(L_\mu f- L_\nu f)(s)| = |\int_K L_x f \ d(\mu-\nu)(x)| \leq \int_K ||L_xf||_\infty \ d (|\mu-\nu|)(x) \leq ||f||_\infty \ |\mu-\nu|(K)$$ and similarly $$|L_\nu f(s)| \leq \int_K ||L_xf||_\infty \ d|\nu|(x) \leq ||f||_\infty \ |\nu|(K).$$ Hence $\Psi$ is jointly continuous.

Next, we choose some $v_0\in AP(K)^*$ such that $v_0(\mathbf{1}) = 1$. Then for any $\mu\in M(K)$ we have that $$L_\mu^*v_0 (\mathbf{1}) = v_0(L_\mu (\mathbf{1}) = v_0(||\mu|| \mathbf{1}) = ||\mu||.$$
Hence in particular, we have that $\big{(}L_{p_s}^*v_0 - v_0\big{)} \in \mathbb{C}\mathbf{1}^{\perp} \simeq X^*$. We consider the map $\pi:K \times X^*\ra X^*$ given by $\pi(s, u) := T_su$ for each $s\in K, u\in X^*$ where $T_su$ is given by $$ T_su := L_{p_s}^*(u + v_0) -v_0 .$$

Now in order to prove that $\pi$ indeed is an action of $K$ on $X^*$, we derive the following equality for any $s, t\in K$, $u\in X^*$ and for each $f\in X$.
\begin{eqnarray*}
T_s(T_tu)(f) &=& L_{p_s}^*(T_tu)(f) + L_{p_s}^* v_0(f) -v_0(f)\\
&=& T_tu(L_{p_s}f) + v_0(L_{p_s} f) - v_0(f)\\
&=& T_tu(L_sf) + v_0(L_s f) - v_0(f)\\
&=& \Big{(}L_{p_t}^*u(L_sf) + L_{p_t}^*v_0(L_sf) - v_0(L_sf)\Big{)} + v_0(L_s f) - v_0(f)\\
&=& (u+v_0)(L_t(L_sf))  - v_0(f)\\
&=& (u+v_0)(L_{(p_s*p_t)} f) - v_0(f)\\
&=& (u+v_0) \Big{(} \int_K L_\zeta f \ d(p_s*p_t)(\zeta)\Big{)} -v_0(f) \\
&=& \int_K \Big{(} u(L_\zeta f) + v_0 (L_\zeta f)\Big{)} \ d(p_s*p_t)(\zeta) - v_0(f) \int_K \ d(p_s*p_t)\\
&=& \int_K \Big{(} u(L_{p_\zeta} f) + v_0 (L_{p_\zeta} f) - v_0(f)\Big{)} \ d(p_s*p_t)(\zeta)\\
&=& \int_K \Big{(} L_{p_\zeta}^*u (f) + L_{p_\zeta}^*v_0(f)) - v_0(f)\Big{)} \ d(p_s*p_t)(\zeta)\\
&=& \int_K T_\zeta(f) \ d(p_s*p_t)(\zeta).
\end{eqnarray*}
Thus we see that $\pi$ is a jointly continuous affine representation of $K$ on $X^*$. Moreover, fixing any $u_0\in X^*$  we have that 
\begin{eqnarray*}
|T_su_0(f)| &=& |L_{p_s}^*(u_0+v_0)(f) - v_0(f)|\\
&=& |(u_0 +v_0) (L_sf) - v_0(f)|\\
&\leq & \big{(} ||u_0 + v_0|| + ||v_0||\big{)} \ ||f||_\infty,
\end{eqnarray*}
for each $s\in K, f\in X$, i.e, the orbit $\mathcal{O}(u_0)$ is bounded. Hence the given fixed point property ensures that $\pi$ has a common fixed point $w_0\in X^* \simeq \mathbb{C}\mathbf{1}^{\perp}$, i.e,  $$w_0 = T_sw_0= L_{p_s}^*(w_0+ v_0) - v_0$$ for each $s\in K$. In particular, setting $m_0 := w_0 + v_0$ we have that $$m_0(L_sf) = m_0(L_{p_s}f)= (w_0+ v_0)(L_{p_s}f) = L_{p_s}^*(w_0+ v_0) = w_0+v_0 = m_0,$$
for each $s\in K, f\in AP(K)$. Since $w_0(\mathbf{1})=0$, it immediately follows from the choice of $v_0$ that $m_0(\mathbb{1}) =1$. Finally, it is immediate from the duality property of functionals that $m_0$ is indeed a mean, and hence a LIM on $AP(K)$.

The converse of the statement follows as an application of the Krein- Milman theorem, using the boundedness of the orbit.
\end{proof}

\begin{remark}
Note that although these characterizations answer certain open questions affirmatively in terms of the interplay between amenability and fixed points of actions in the setting of coset spaces, orbit space, hypergroups and in general semitopological semihypergroups, there are still many open questions which remain unanswered. For example, it is not known if similar characterizations of amenability is possible for non-affine actions of a semihypergroup in this setting. Furthermore, it is not known whether the existence of a bounded orbit in Theorem \ref{mt3} can be replaced by some other equivalent condition on the representation, which is easier to check and more conventional in nature. Finally, it is only natural that the common fixed points of such actions would be intrinsically connected with the concept of $F$-algebraic amenability. Other that the interplays investigated in \cite{CB4, LZ} for certain kinds of general actions, this area remains unexplored so far.
\end{remark}


\section*{Acknowledgement}

\noi The author would like to sincerely thank her doctoral thesis advisor  Dr. Anthony To-Ming Lau, for suggesting the topic of this study and for the helpful discussions during the initiation of this work. She would also like to gratefully acknowledge the financial support provided by the Indian Institute of Technology Kanpur, India, Harish-Chandra Research Institute, India and SRM University AP, India  during the preparation of this manuscript. 

\end{document}